\documentclass[a4paper,11pt,twoside]{article}
\usepackage{amssymb,amsmath,latexsym,geometry,fancyhdr,lineno,hyperref,titletoc}
\contentsmargin{0pt}

\dottedcontents{section}[2em]{\vspace{-1mm}\small}{2em}{0pt}
\dottedcontents{subsection}[5em]{\vspace{-1mm}\small}{3em}{5pt}

\hypersetup{colorlinks,linkcolor= blue,citecolor=blue}
\newtheorem{theorem}{Theorem}[section]

\newtheorem{lemma}[theorem]{Lemma}
\newtheorem{proposition}[theorem]{Proposition}
\newtheorem{remark}[theorem]{Remark}

\numberwithin{equation}{section}
\allowdisplaybreaks[0]

\begin{document}
\title{{\bf\Large Multiple solutions to
a class of $p$-Laplacian Schr\"{o}dinger equations}}
\author{\\
{ \textbf{\normalsize Lin Zhang}}\\
{\it\small College of Mathematics and Statistics,}\\
{\it\small Xinyang Normal University,}\\
{\it\small Henan 464000, P.R.China}\\
{\it\small e-mail address:
linzhang@xynu.edu.cn}
}
\date{}
\maketitle
{\bf\normalsize Abstract.} {\small
In this paper, we will prove the existence of infinitely many solutions to the following equation by utilizing the variational perturbation method
\begin{equation*}
-div(A(x,u)|\nabla u|^{p-2}\nabla u)+\frac{1}{p}A_{t}(x,u)|\nabla u|^{p}+V(x)|u|^{p-2}u=g(x, u), ~~ x\in\mathbb{R}^{N},
\end{equation*}
where $2<p<N$, $V(x)$ represents a coercive potential. It should be emphasized that the multiplicity of solutions is unsuitable to be discussed under the framework of $W^{1,p}(\mathbb{R}^{N})\cap L^{\infty}(\mathbb{R}^{N})$, as the functional corresponding to the above equation does not satisfy the Palais-Smale condition in this space. To this end, we will use the variational perturbation method to construct novel perturbation space and perturbation functional. By relying on the established conclusions regarding the multiplicity of solutions for classical $p$-Laplacian equations, we shall analyze the equation in question. Our result addresses positively the open question from Candela et. al in \cite{CSS22CVPDE}.}



\let\thefootnote\relax\footnotetext{{ \footnotesize This work was supported by the Young Project of Henan Natural Science Foundation(Grant No.252300420928)}}

\pagestyle{fancy}
\fancyhead{} 
\fancyfoot{} 
\renewcommand{\headrulewidth}{0pt}
\renewcommand{\footrulewidth}{0pt}
\fancyhead[CE]{ \textsc{Lin Zhang}}
\fancyhead[CO]{ \textsc{Multiple solutions to
a class of $p$-Laplacian Schr\"{o}dinger equations}}
\fancyfoot[C]{\thepage}


\newpage
\section{Introduction}

~~~~In this paper, our main focus lies on investigation of the existence of multiple solutions for the following  non-autonomous $p$-Laplacian type quasi-linear equations in the whole space $\mathbb{R}^{N}$
\begin{equation}\label{PL}
-div(A(x,u)|\nabla u|^{p-2}\nabla u)+\frac{1}{p}A_{t}(x,u)|\nabla u|^{p}+V(x)|u|^{p-2}u=g(x, u).
\end{equation}

The notable feature here is that $A$ depends on $u$. In case $p=2$, $A(u)=1+u^{2}$ and $g(x,u)=|u|^{q-2}u$, this is referred as Modified Nonlinear Schr\"{o}dinger Equation
\begin{equation}\label{2L}
\left\{
\begin{aligned}
&  \Delta u-V(x)u+\frac{1}{2}u\Delta(u^{2})+|u|^{q-2}u=0, ~~in ~~\Omega\\
& u=0, ~~on~~\Omega.
\end{aligned}
\right.
\end{equation}

In recent years, there are three important methods in studying the existence and multiplicity of solutions to the quasi-linear Schr\"{o}dinger equations: the constrained minimization method, the dual approach, and the variational perturbation method. The above three methods can be used to overcome the inevitable "common problem" when solving this class of quasi-linear Schr\"{o}dinger equations. The "common problem" is that the variational functional corresponding to the quasi-linear equations is non-differentiable. In fact, when overcoming this issue, all three methods mentioned above have their own advantages and limitations. For the quasi-linear functionals, the constrained minimization method can "avoid" its non-differentiability and analyze it directly. Unlike it, the dual approach and the variational perturbation method aim to recover the smoothness of the quasi-linear functionals. It should be pointed out that the change of variable method was first proposed in \cite{LWW03JDE}. Afterwards, in \cite{CJ04NA}, Colin and Jeanjean called it the dual approach. This method can be used to transform the quasi-linear equation to a semi-linear one. However, it is not applicable to the functionals with the general quasi-linear form $a_{ij}(x,u)\partial_{i}u\partial_{j}u$. For this motivation, another powerful approach was initiated in a series of papers \cite{LLW13PAMS,LLW13JDE,LW14JDE}, so called the variational perturbation method. This method is applicable to general quasi-linear equations with the above form and is effective for establishing multiple solutions in particular infinitely many solutions in the setting of the symmetric mountain pass theorem.

For the non-autonomous $p$-Laplacian type quasi-linear equation (\ref{PL}), if $A_{t}(x,u)$ is not identical to zero, in \cite{CS20NA}, Candela and Salvatore proved the existence of radial bounded solutions for (\ref{PL}) when $V(x)=1$. In \cite{CSS22CVPDE}, Candela, Salvatore and Sportelli proved the existence of bounded solutions for equation (\ref{PL}) in the whole space $\mathbb{R}^{N}$. In their proof, a bounded domain approximation was used. First, with the help of the Mountain Pass Theorem, the existence of bounded solutions to the objective equation in bounded domains is proved. Then, by extending the solutions obtained in the bounded domains, the existence of bounded solutions in the whole space $\mathbb{R}^{N}$ was proved. It should be pointed out that the method they used is only applicable to proving the existence of solutions. Moreover, Candela et al also raised a question in \cite{CSS22CVPDE}: can suitable perturbations be identified to apply the variational perturbation method and establish the multiplicity of solutions for such non-autonomous $p$-Laplacian equations (\ref{PL}) ?

The motivation of the current paper is to address this open question. For this purpose, we will use the variational perturbation method to prove the existence of infinitely many solutions to equation (\ref{PL}) when $2<p<N$ and $V(x)$ is coercive potential. It should be pointed out that when $p=2$, the potential function $V(x)=0$, $A(x,u)$ is constant and $g(x,u)=|u|^{q-2}u$ in equation (\ref{PL}), it can be transformed into a classical 2-laplacian quasi-linear elliptic equation.

Now, we define the weak solution of equation (\ref{PL}): $u\in W_{V}^{1,p}(\mathbb{R}^{N})\cap L^{\infty}(\mathbb{R}^{N})$ is said to be the weak solution of equation (\ref{PL}) if and only if $\langle dI(u),\phi\rangle=0$ holds for any $\phi\in C^{\infty}_{0}(\mathbb{R}^{N})$. The definitions of $W^{1,p}_{V}(\mathbb{R}^{N})$ and $L^{\infty}(\mathbb{R}^{N})$ will be given below. Here, the functional $I(u)$  corresponding to the equation (\ref{PL}) is defined as
$$I(u)=\frac{1}{p}\int_{\mathbb{R}^{N}}A(x,u)|\nabla u|^{p}dx+\frac{1}{p}\int_{\mathbb{R}^{N}}V(x)|u|^{p}dx-\int_{\mathbb{R}^{N}}G(x,u)dx,$$
where $G(x,t)=\int_{0}^{t}f(x,s)ds$, and the representation of $\langle dI(u),\phi\rangle$ is as follows:
for any $\phi\in C^{\infty}_{0}(\mathbb{R}^{N})$,
$$\langle dI(u),\phi\rangle=\int_{\mathbb{R}^{N}}A(x,u)|\nabla u|^{p-2}\nabla u\nabla\phi+\frac{1}{p}A_{t}(x,u)\phi|\nabla u|^{p}
+V(x)|u|^{p-2}u\phi-g(x,u)\phi dx.$$

Next, we provide the relevant assumptions for $A(x,t)$, $V(x)$ and $g(x,t)$:
\begin{enumerate}
  \item[($V_{1}$).] $V(x)\in C(\mathbb{R}^{N}, \mathbb{R})$, $\inf_{x\in\mathbb{R}^{N}}V(x)>0$.
  \item[($V_{2}$).] $\lim_{|x|\rightarrow+\infty}V(x)=\infty$.
  \item[($h_{0}$).] $A(x,t)=a_{1}+a_{2}|t|^{\alpha p}$, where $\alpha p>1$, $a_{1}$, $a_{2}$ are positive constant.
  \item[($g_{0}$).] $g(x,t)\in C(\mathbb{R}^{N}\times\mathbb{R},\mathbb{R})$, and is odd with respect to $t$.
  \item[($g_{1}$).] $\lim_{t\rightarrow 0}\frac{g(x,t)}{|t|^{p-1}}=0$, $\lim_{t\rightarrow \infty}\frac{g(x,t)}{|t|^{p-1}}=+\infty$. And there exist $q\in(p(1+\alpha),\frac{Np(1+\alpha)}{N-p})$ and $C>0$, such that $g(x,t)\leq C(1+|t|^{q-1})$.
  \item[($g_{2}$).]There exists $\mu\in(p(1+\alpha),\frac{Np(1+\alpha)}{N-p})$,  such that $0<\mu G(x,t)\leq g(x,t)t$ for any $x\in\mathbb{R}^{N}$, $t\neq0$.
\end{enumerate}

Now, for the existence of multiple solutions of quasi-linear equation (\ref{PL}),we have the following theorem.
\begin{theorem}\label{THM1.1} \rm{Let $2<p<N$. Assume that the above conditions ($V_{1}$), ($V_{2}$), ($h_{0}$) and ($g_{0}$)-($g_{2}$) hold. Then the equation $(\ref{PL})$ has infinitely many solutions.}
\end{theorem}

When $A(x,u)$ remains constant, it can be expressed as the following classical $p$-Laplacian equation
\begin{equation}\label{CPL}
-\Delta_{p}u+V(x)|u|^{p-2}u=g(x,u),
\end{equation}
where $\Delta_{p}u:=div(|\nabla u|^{p-2}\nabla u)$. Actually, there are many studies on the existence and multiplicity of solutions for the classical $p$-Laplacian equations. Such as, in bounded domains, the studies on the existence of solutions to the classical $p$-Laplacian can refer to \cite{BL04JDE,CM95NA,FL00NA,LZ01NA}, the study of multiplicity can be refer to \cite{AAP96JFA,BCS13NoDEA,BCS13DCDS,BLW05PLMS,CP06ANS,CPY12JFA,HP07MM,JS03JMAA,LL15AMAS,LZ02JLMS}. And in the unbounded region, there are also some studies on the existence of solutions to the classical $p$-Laplacian equations, which can be referred to \cite{CM95JMAA,Y92PAMS}. It should be pointed out that in reference \cite{BCS13NoDEA}, it mentioned that when $p=2$, the spectrum of $-\Delta$ is composed of a sequence of divergent eigenvalues $(\lambda_{k})_{k}$, but the current research on the spectral analysis of $-\Delta_{p}$ is not comprehensive. This is also one of the difficulties we need to overcome when establishing the existence of infinitely many solutions to equation (\ref{PL}).

Besides that, there are also many studies on the classical $p$-Laplacian equation in the whole space $\mathbb{R}^{N}$. For instance, in \cite{SL03MN}, Su and Lin proved that the classical non-autonomous $p$-Laplacian equation has  at least one positive solution and one negative solution. Alves and Figueiredo proved in \cite{AF06DIE} that when the potential function satisfies suitable conditions,  the $p$-Laplacian equation has multiple positive solutions with small parameter. In \cite{D09NA}, Dai proved that the  $p(x)$-Laplacian equation has infinitely many solutions. In \cite{LZ11JMAA}, Lin and Zheng demonstrated the existence of non-trivial solutions to the $p(x)$-Laplacian equation with potential function. Afterwards, similar to the conditions mentioned in \cite{LZ11JMAA},  Lin and Tang verified the existence of infinitely many solutions to the $p(x)$-Laplacian equation in \cite{LT13NA}. Furthermore, for the scenario in which the potential function is bounded, the existence of multiple solutions was proved through truncation in references \cite{LZ17ANS,ZLL17JMAA}. Of course, there are still many studies on the solutions of the classical $p$-Laplacian equation, relevant references can refer to the literature mentioned above.

 For the sake of readability, the paper is organized as follows. In Section 2, we provide some necessary symbol representations and the symmetry mountain path theorem. In Section 3, we mainly divide into three steps to prove the existence of multiple solutions to the equation $(\ref{PL})$. Firstly, we analyze the continuous differentiability of perturbation functional by constructing perturbation space in Proposition \ref{PRO3.2}. Subsequently, we prove the convergence theorem for the perturbation functional in Lemma \ref{LEM3.4}. Secondly, we prove the perturbation functional satisfies the Palais-Smale condition in Lemma \ref{LEM3.5}. Finally, we will use the symmetric mountain pass theorem and some technical analysis to prove Theorem \ref{THM1.1}.


\section{Preliminary}

~~~~For the convenience of the following proof, we will provide some relevant notations and concepts below.
\begin{itemize}
  \item $L^{p}(\mathbb{R}^{N})$ denotes the classical Lebesgue space with norm $\|u\|_{L^{p}}:=(\int_{\mathbb{R}^{N}}|u|^{p}dx)^{\frac{1}{p}}$, where $1\leq p<+\infty$.
  \item $L^{p}_{V}(\mathbb{R}^{N})$ denotes the weighted Lebesgue space with norm $\|u\|_{V,p}:=(\int_{\mathbb{R}^{N}}V(x)|u|^{p}dx)^{\frac{1}{p}}$.
  \item $W^{1,p}(\mathbb{R}^{N})$ denote the classical Sobolev space with norm $\|u\|_{W^{1,p}}:=(\|\nabla u\|_{p}^{p}+\|u\|_{p}^{p})^{\frac{1}{p}}$.
  \item $W^{1,p}_{V}(\mathbb{R}^{N}):=\{u\in W^{1,p}(\mathbb{R}^{N})~| ~ \int_{\mathbb{R}^{N}}V(x)|u|^{p}<\infty\}$ denote the weighted Sobolev space with norm $\|u\|_{W^{1,p}_{V}}=(\|\nabla u\|_{p}^{p}+\|u\|^{p}_{V,p})^{\frac{1}{p}}$.
  \item $L^{\infty}(\mathbb{R}^{N}):=\{u:\mathbb{R}^{N}\rightarrow\mathbb{R}~ is ~a ~measurable ~function~ |~\textrm{ess~sup}_{\mathbb{R}^{N}}|u(x)|<\infty\}$ with the norm $\|u\|_{L^{\infty}}:=\inf\{C ~| ~\,|u(x)|\leq C ~for ~a.e.~ x\in\mathbb{R}^{N} \}$.
\end{itemize}
\textbf{Symmetric Mountain Pass Theorem}(See \cite{R65CBMS}) Let $E$ be an infinite dimensional Banach space,
$I\in C^{1}(E,\mathbb{R}^{N})$ be an even functional that satisfies the Palais-Smale condition[short for (PS)], and $I(0)=0$. If $E=F\oplus W$, where $F$ is a finite dimensional space, and $I$ satisfies
\begin{enumerate}
  \item[(1).] for each finite dimensional subspace $\tilde{E}\subset E$, there is an $R=R(\tilde{E})>0$, such that $I(u)<0$ on $\tilde{E}/B_{R(\tilde{E})}$,
  \item[(2).] there are constants $\rho,\alpha>0$ such that $I_{\partial B_{\rho}\cap W}\geq\alpha$.
\end{enumerate}
Then $I$ possesses an unbounded sequence of critical values.


\section{The proof of multiple solutions}

~~~~To prove the existence of multiple solutions for the non-autonomous $p$-Laplacian equation (\ref{PL}), we will proceed with the following steps: firstly, we shall analyze the continuous differentiability of the perturbation functional $I_{\nu}(u)$ by constructing the perturbation space $X:=W^{1, r}(\mathbb{R}^{N})\cap W_{V}^{1,p}(\mathbb{R}^{N})$. Based in this, the important convergence lemma(Lemma \ref{LEM3.4}) concerning the perturbation functional $I_{\nu}(u)$ can be established. Secondly, we will prove that the perturbation functional satisfies the Palais-Smale condition. Finally, the existence of infinite solutions of equation (\ref{PL}) can be proved by using the symmetric mountain pass theorem and some related analyzes.
\subsection{The continuous differentiability of the perturbation functional $I_{\nu}(u)$}

~~~~Since the functional (\ref{pf}) corresponding to the $p$-laplacian equation (\ref{PL}) is not $C^{1}$ on $W_{V}^{1,p}(\mathbb{R}^{N})$.
\begin{equation}\label{pf}
I(u)=\frac{1}{p}\int_{\mathbb{R}^{N}}A(x,u)|\nabla u|^{p}dx+\frac{1}{p}\int_{\mathbb{R}^{N}}V(x)|u|^{p}dx-\int_{\mathbb{R}^{N}}G(x,u)dx
\end{equation}
In order to overcome the non-smooth difficulty of the quasi-linear functional $I(u)$, we will consider its perturbation functional (\ref{ppf}) in space $X:=W^{1, r}(\mathbb{R}^{N})\cap W_{V}^{1,p}(\mathbb{R}^{N})$, here we set $r=(1+\alpha)p$.
\begin{align}\label{ppf}
I_{\nu}(u)=\frac{\nu}{2}\int_{\mathbb{R}^{N}}(|\nabla u|^{r}+|u|^{r})dx+I(u),
\end{align}
where $\nu\in(0,1]$. $u$ is called a solution to the following formula (\ref{pws}) if and only if $\langle dI_{\nu}(u),\phi\rangle=0$ for all $\phi\in X$, i.e. $u$ satisfies the following form
\begin{align}\label{pws}
&\nu\int_{\mathbb{R}^{N}}(|\nabla u|^{r-2}\nabla u\nabla \phi+|u|^{r-2}u\phi)dx+\langle dI(u),\phi\rangle\nonumber\\
=&\nu\int_{\mathbb{R}^{N}}(|\nabla u|^{r}\nabla u\nabla \phi+|u|^{r}u\phi)dx+\int_{\mathbb{R}^{N}}A(x,u)|\nabla u|^{p-2}\nabla u\nabla\phi dx\nonumber\\
&~~+\frac{1}{p}\int_{\mathbb{R}^{N}}A_{t}(x,u)\phi|\nabla u|^{p}dx+\int_{\mathbb{R}^{N}}V(x)|u|^{p-2}u\phi dx-\int_{\mathbb{R}^{N}}g(x,u)\phi dx=0.
\end{align}
\begin{lemma}\label{LEM3.1}\rm{
For any $x,y\in\mathbb{R}^{m}$, $m\geq1$, there exists a constant $C_{0}>0$, such that
\begin{equation}\label{bd1}
||x|^{s-2}x-|y|^{s-2}y|\leq C_{0}|x-y|(|x|+|y|)^{s-2}, ~for~ s>2.
\end{equation}
\begin{equation}\label{bd2}
||x|^{s-2}x-|y|^{s-2}y|\leq C_{0}|x-y|^{s-1}, ~for ~1<s\leq2.
\end{equation}}
\end{lemma}
Next, we will prove that the perturbation functional $I_{\nu}(u)$ is continuously differentiable on the perturbation space $X$ by the following proposition.
\begin{proposition}\label{PRO3.2}\rm{
For $2<p<N$, suppose that the condition $(V_{1})$, $(h_{0})$ and $(g_{0}),(g_{1})$ given in Theorem \ref{THM1.1} are hold. If $\{u_{n}\}$ converges strongly to $u$ in
$X:=W^{1,r}(\mathbb{R}^{N})\cap W^{1,p}_{V}(\mathbb{R}^{N})$, then
$I_{\nu}(u_{n})\rightarrow I_{\nu}(u)$ and
$\|dI_{\nu}(u_{n})-dI_{\nu}(u)\|_{X^{'}}\rightarrow 0$ as $n\rightarrow\infty$, i.e., $I_{\nu}(u)$ is a $C^{1}$ functional in the perturbation space $X$.}
\end{proposition}
\begin{proof}
For convenience, we will write the perturbation functional (\ref{ppf}) as
$I_{\nu}(u)=A_{1}(u)+A_{2}(u)+A_{3}(u)-A_{4}(u)$. Here,
$$A_{1}(u)=\frac{\nu}{2}\int_{\mathbb{R}^{N}}(|\nabla u|^{r}+|u|^{r})dx,~~A_{2}(u)=\frac{1}{p}\int_{\mathbb{R}^{N}}A(x,u)|\nabla u|^{p}dx,$$
$$A_{3}(u)=\frac{1}{p}\int_{\mathbb{R}^{N}}V(x)|u|^{p}dx,~~ A_{4}(u)=\int_{\mathbb{R}^{N}}G(x,u)dx.$$
The G$\hat{a}$teaux differentials corresponding to the above formulas are
\begin{align*}&\langle dA_{1}(u),\phi\rangle
=\nu\int_{\mathbb{R}^{N}}(|\nabla u|^{r-2}\nabla u\nabla \phi+|u|^{r-2}u\phi)dx,
\end{align*}
$$\langle dA_{2}(u),\phi\rangle=\int_{\mathbb{R}^{N}}A(x,u)|\nabla u|^{p-2}\nabla u\nabla\phi dx+\frac{1}{p}\int_{\mathbb{R}^{N}}A_{t}(x,u)\phi|\nabla u|^{p}dx,$$
$$\langle dA_{3}(u),\phi\rangle=\int_{\mathbb{R}^{N}}V(x)|u|^{p-2}u\phi dx,~~\langle dA_{4}(u),\phi\rangle=\int_{\mathbb{R}^{N}}g(x,u)\phi dx.$$
Now, we demonstrate the continuous differentiability of the perturbation functional $I_{\nu}(u)$ in four steps:\\
\textbf{Step 1:} we need to prove that $A_{1}(u_{n})\rightarrow A_{1}(u)$ and $\|dA_{1}(u_{n})-dA_{1}(u)\|_{X^{'}}\rightarrow 0$ as $n\rightarrow\infty$.\\
Due to $\|u_{n}-u\|_{X}\rightarrow0$, we have $A_{1}(u_{n})\rightarrow A_{1}(u)$. To prove $\|dA_{1}(u_{n})-dA_{1}(u)\|_{X^{'}}\rightarrow 0$, it need to prove
$$\sup_{\|\phi\|_{X}=1}|\langle dA_{1}(u_{n})-dA_{1}(u),\phi\rangle|\rightarrow 0.$$
Here, \begin{align*}
&\sup_{\|\phi\|_{X}=1}|\langle dA_{1}(u_{n})-dA_{1}(u),\phi\rangle|\\
=&\sup_{\|\phi\|_{X}=1}|\nu
\int_{\mathbb{R}^{N}}(|\nabla u_{n}|^{r-2}\nabla u_{n}-|\nabla u|^{r-2}\nabla u)\nabla \phi +(|u_{n}|^{r-2}u_{n}-|u|^{r-2}u)\phi dx|\\
\leq&\sup_{\|\phi\|_{X}=1}|\nu\int_{\mathbb{R}^{N}}(|\nabla u_{n}|^{r-2}\nabla u_{n}-|\nabla u|^{r-2}\nabla u)\nabla \phi dx|\\
&\qquad+\sup_{\|\phi\|_{X}=1}|\nu\int_{\mathbb{R}^{N}}(|u_{n}|^{r-2}u_{n}-|u|^{r-2}u)\phi dx|\\
=:&~~\mathbb{I}+\mathbb{II}.
\end{align*}
Next, we prove that $\mathbb{I}\rightarrow0$ and $\mathbb{II}\rightarrow0$ as $n\rightarrow\infty$. It should be pointed out that the proof of $\mathbb{II}$ is similar to $\mathbb{I}$. To avoid repetition, we just prove that $\mathbb{I}\rightarrow0$ as $n\rightarrow\infty$. Since $p>2$ and $\alpha p>1$, we have $r=(1+\alpha)p>2$. Then according to Lemma \ref{LEM3.1}, $\|u_{n}-u\|_{X}\rightarrow0$  and H\"{o}lder inequality, we have
\begin{align*}
\mathbb{I}&\leq\sup_{\|\phi\|_{X}=1}|\int_{\mathbb{R}^{N}}|\nabla u_{n}-\nabla u||\nabla u_{n}+\nabla u|^{r-2}|\nabla\phi| dx|\\
&\leq[\int_{\mathbb{R}^{N}}|\nabla u_{n}-\nabla u|^{r}dx]^{\frac{1}{r}}[\int_{\mathbb{R}^{N}}|\nabla u_{n}+\nabla u|^{r}dx]^{\frac{r-2}{r}}\rightarrow0.
\end{align*}
In the same way, the proof of $\mathbb{II}\rightarrow0$ is easy to prove. Thus, $\|dA_{1}(u_{n})-dA_{1}(u)\|_{X^{'}}\rightarrow 0$.\\
\textbf{Step 2:} The proof of $A_{2}(u_{n})\rightarrow A_{2}(u)$ and $\|dA_{2}(u_{n})-dA_{2}(u)\|_{X^{'}}\rightarrow 0$ as $n\rightarrow\infty$.\\
Under condition ($h_{0}$), according to Lemma \ref{LEM3.1} and H\"{o}lder inequality, we have
\begin{align*}
A_{2}(u_{n})-A_{2}(u)=&\frac{1}{p}\int_{\mathbb{R}^{N}}
(a_{1}+a_{2}|u_{n}|^{p\alpha}) (|\nabla u_{n}|^{p}-|\nabla u|^{p})dx\\
&\qquad+\frac{1}{p}\int_{\mathbb{R}^{N}}a_{2}(|u_{n}|^{p\alpha}-|u|^{p\alpha})|\nabla u|^{p}dx\\
\rightarrow&0,
\end{align*}
and
\begin{align*}
&\sup_{\|\phi\|_{X}=1}|\langle dA_{2}(u_{n})-dA_{2}(u),\phi\rangle|\\
\leq&\sup_{\|\phi\|_{X}=1}|\int_{\mathbb{R}^{N}}(a_{1}
+a_{2}|u_{n}|^{p\alpha}) (|\nabla u_{n}|^{p-2}\nabla u_{n}- |\nabla u|^{p-2}\nabla u)\nabla \phi dx|\\
&+a_{2}\sup_{\|\phi\|_{X}=1}|\int_{\mathbb{R}^{N}}(|u_{n}|^{p\alpha}-|u|^{p\alpha})|\nabla u|^{p-1}\nabla\phi dx|\\
&+\alpha a_{2}\sup_{\|\phi\|_{X}=1}|\int_{\mathbb{R}^{N}}|u_{n}|^{p\alpha-1}(|\nabla u_{n}|^{p}-|\nabla u|^{p})\phi dx|\\
&+\alpha a_{2}\sup_{\|\phi\|_{X}=1}|\int_{\mathbb{R}^{N}}(|u_{n}|^{p\alpha-1}-|u|^{p\alpha-1})|\nabla u|^{p}\phi dx|\rightarrow0.
\end{align*}
\textbf{Step 3:} In conditions $(g_{0}), (g_{1})$, according to the Dominated Convergence Theorem and H\"{o}lder inequality, we have $A_{4}(u_{n})\rightarrow A_{4}(u)$ and $\|dA_{4}(u_{n})-dA_{4}(u)\|_{X^{'}}\rightarrow 0$ as $n\rightarrow0$.\\
\textbf{Step 4:} Now we prove that $A_{3}(u_{n})\rightarrow A_{3}(u)$ and $\|dA_{3}(u_{n})-dA_{3}(u)\|_{X^{'}}\rightarrow 0$ as $n\rightarrow0$.\\
Following from $\|u_{n}-u\|_{W^{1,p}_{V}}\rightarrow0$, we can deduce that $A_{3}(u_{n})\rightarrow A_{3}(u)$ easily. For $\sup_{\|\phi\|_{X}=1}|\langle dA_{3}(u_{n})-dA_{3}(u),\phi\rangle|$, since $p>2$, then by Lemma \ref{LEM3.1} and H\"{o}lder inequality, we have
\begin{align*}&\sup_{\|\phi\|_{X}=1}|\langle dA_{3}(u_{n})-dA_{3}(u),\phi\rangle|\\
\leq&\sup_{\|\phi\|_{X}=1}|\int_{\mathbb{R}^{N}}|V(x)|^{\frac{1}{p}}
|V(x)|^{\frac{p-1}{p}}|u_{n}^{p-1}-u^{p-1}||\phi|dx|\\
\leq&\sup_{\|\phi\|_{X}=1}|(\int_{\mathbb{R}^{N}}|V(x)||u_{n}-u|^{\frac{p}{p-1}}|u_{n}+u|^{\frac{p(p-2)}{p-1}}dx)^{\frac{p-1}{p}}
(\int_{\mathbb{R}^{N}}|V(x)||\phi|^{p}dx)^{\frac{1}{p}}|\\
\leq&(\int_{\mathbb{R}^{N}}|V(x)||u_{n}-u|^{p}dx)^{\frac{1}{p}}
(\int_{\mathbb{R}^{N}}|V(x)||u_{n}+u|^{p}dx)^{\frac{p-2}{p}}\\
\rightarrow &0.
\end{align*}
In summary, it can be concluded that the perturbation functional $I_{\nu}(u)$ is continuously differentiable on the perturbation space $X:=W^{1,r}(\mathbb{R}^{N})\cap W^{1,p}_{V}(\mathbb{R}^{N})$.
\end{proof}
\begin{remark}\label{REK3.1}\rm{If $\alpha p\geq0$, the perturbation functional $I_{\nu}(u)$ is continuous in $W^{1,p(1+\alpha)}\cap W^{1,p}_{V}(\mathbb{R}^{N})$. To ensure differentiability, $\alpha p>1$ is also required.}
\end{remark}
Next, in the case where the perturbation functional $I_{\nu}(u)$ is continuously differentiable, we will prove the following important convergence lemma.
\begin{lemma}\label{LEM3.4}\rm{Under conditions $(V_{1}), ~(V_{2})$, $(h_{0})$ and $(g_{0})-(g_{2})$, assume that
$I_{\nu_{n}}(u_{n})\leq C$, \,$I_{\nu_{n}}^{'}(u_{n})=0$, and $\nu_{n}\rightarrow0$($n\rightarrow\infty$) for $\{u_{n}\}\subset X$. Then there exists $u\in W^{1,p}_{V}(\mathbb{R}^{N})\cap L^{\infty}(\mathbb{R}^{N})$ as the weak solution of equation (\ref{PL}), and for the subsequence of $\{u_{n}\}$, still denote as $\{u_{n}\}$, it satisfies $\|u_{n}-u\|_{W^{1,p}_{V}}\rightarrow 0$, $\|(\nabla u_{n})|u_{n}|^{\alpha}-(\nabla u)|u|^{\alpha}\|_{L^{p}}\rightarrow0$, $\nu_{n}\|u_{n}\|^{r}_{W^{1,r}}\rightarrow 0$ and $I(u)=\lim_{n\rightarrow\infty}I_{\nu_{n}}(u_{n})$.}
\end{lemma}
\begin{proof}
Now, we will prove this lemma in four steps.\\
\textbf{Step 1:} It need to prove that $\{u_{n}\}$ is bounded in $W^{1,p}_{V}(\mathbb{R}^{N})$, and $\{|u_{n}|^{\alpha}\nabla u_{n}\}$ is bounded in $L^{p}(\mathbb{R}^{N})$. \\
Following from conditions $(h_{0})$ and $(g_{2})$, there exists a constant $\beta_{0}>0$ satisfies
$$(\mu-p)A(x,t)-A_{t}(x,t)t\geq\beta_{0}A(x,u).$$ According to conditions $(V_{1})$, $(h_{0})$ and $(g_{2})$, we have
\begin{align*}
&I_{\nu}(u)-\frac{1}{\mu}\langle I^{'}_{\nu}(u),u\rangle\\
=&(\frac{\mu-2}{2\mu})\nu\int_{\mathbb{R}^{N}}(|\nabla u|^{r}+|u|^{r})dx
+\frac{1}{p\mu}\int_{\mathbb{R}^{N}}[(\mu-p)A(x,u)-A_{t}(x,u)u]|\nabla u|^{p}dx\\
&+\frac{\mu-p}{p\mu}\int_{\mathbb{R}^{N}}V(x)|u|^{p}dx
+\frac{1}{\mu}\int_{\mathbb{R}^{N}}[g(x,u)u-\mu G(x,u)]dx\\
\geq&(\frac{\mu-2}{2\mu})\nu\int_{\mathbb{R}^{N}}(|\nabla u|^{r}+|u|^{r})dx+\beta_{0}\int_{\mathbb{R}^{N}}[a_{1}(x)+a_{2}(x)|u|^{\alpha p}]|\nabla u|^{p}dx\\
&+\frac{\mu-p}{p\mu}\int_{\mathbb{R}^{N}}V(x)|u|^{p}dx.
\end{align*}
Since $I_{\nu_{n}}(u_{n})\leq C$, we have $\{u_{n}\}$ is bounded in $W^{1,p}_{V}(\mathbb{R}^{N})$ and $\{|u_{n}|^{\alpha}\nabla u_{n}\}$ is bounded in $L^{p}(\mathbb{R}^{N})$. Then
\begin{align*}
&u_{n}\rightharpoonup u ~in ~W^{1,p}_{V}(\mathbb{R}^{N}),\\
&u_{n} \rightarrow u ~in~L^{s_{1}}(\mathbb{R}^{N}), ~~ s_{1}\in(p,p^{*}), ~p^{*}=\frac{pN}{N-p},\\
&u_{n} \rightarrow u ~a.e. ~in ~\mathbb{R}^{N},\\
&(\nabla u_{n})|u_{n}|^{\alpha}\rightharpoonup(\nabla u)|u|^{\alpha}~in ~L^{p}(\mathbb{R}^{N}).
\end{align*}
Under conditions $(V_{1})$ and $(V_{2})$, we claim that $\{u_{n}\}$ strongly converges to $u$ in $L^{s}(\mathbb{R}^{N})$ for $s\in[p,(1+\alpha)p^{*})$. In order to prove this claim, we set $w_{n}=u_{n}-u$, if it can be proven that $\{w_{n}\}$ strongly converges to zero in $L^{s}(\mathbb{R}^{N})$, the claim can be proven.
To this end, we need to prove that for $s=p$, $\{w_{n}\}$ strongly converges to zero in $L^{p}(\mathbb{R}^{N})$ first. The details are as follows: according to the equivalence relationship between the weakly convergence and the strongly convergence in bounded regions, we know that $\{w_{n}\}$ strongly converges to zero in $L^{p}(B_{R})$ for any given $R>0$. Here, $B_{R}:=\{x\in\mathbb{R}^{N},\,|x|\leq R\}$. By $(V_{2})$, we know that for any $\epsilon>0$, there exists a constant $R_{1}>0$ such that for any $R\geq R_{1}$, $$\frac{C^{p}}{\inf_{B_{R}^{c}}V(x)}<\frac{\epsilon}{2}.$$
Then, by fixing $R_{1}$, there exists $N_{1}>0$ such that $\int_{B_{R_{1}}}|w_{n}|^{p}\leq\frac{\epsilon}{2}$ for any $n\geq N_{1}$. Thus, for $n\geq N_{1}$, according to the boundedness of $\{u_{n}\}$ in $W^{1,p}_{V}(\mathbb{R}^{N})$, we have
\begin{align*}
\int_{\mathbb{R}^{N}}|w_{n}|^{p}dx&=\int_{B_{R_{1}}}|w_{n}|^{p}dx
+\int_{B_{R_{1}}^{c}}|w_{n}|^{p}dx\\
&\leq\frac{\epsilon}{2}+\frac{1}{\inf_{B_{R_{1}}^{c}}V(x)}\int_{B_{R_{1}}^{c}}V(x)|w_{n}|^{p}dx\\
&\leq \epsilon.
\end{align*}
Therefore, $\{w_{n}\}$ strongly converges to zero in $L^{p}(\mathbb{R}^{N})$. Next, for $s\in(p,(1+\alpha)p^{*})$, according to the boundedness of $\{(\nabla u_{n})|u_{n}|^{\alpha}\}$ in $L^{p}(\mathbb{R}^{N})$, and the interpolation inequality, we have
\begin{align*}
\|w_{n}\|_{s}&\leq\|w_{n}\|_{L^{p}}^{\theta}\|w_{n}\|_{L^{(1+\alpha)p^{*}}}^{1-\theta}\\
&\leq C\|w_{n}\|_{L^{p}}^{\theta}\rightarrow0,
\end{align*}
where $\theta$ satisfies $\frac{1}{s}=\frac{\theta}{p}+\frac{1-\theta}{(1+\alpha)p^{*}}$.\\
\textbf{Step 2:} It need to prove that $\{u_{n}\}$ is uniformly bounded in $L^{\infty}(\mathbb{R}^{N})$.\\
Following from the assumption $I^{'}_{\nu_{n}}(u_{n})=0$, we know that for any $\phi\in X$,
\begin{align*}
&\nu_{n}\int_{\mathbb{R}^{N}}(|\nabla u_{n}|^{r-2}\nabla u_{n}\nabla \phi+|u_{n}|^{r-2}u_{n}\phi)dx+\int_{\mathbb{R}^{N}}A(x,u_{n})|\nabla u_{n}|^{p-2}\nabla u_{n}\nabla\phi dx\\
&~~+\frac{1}{p}\int_{\mathbb{R}^{N}}A_{t}(x,u_{n})\phi|\nabla u_{n}|^{p}dx
+\int_{\mathbb{R}^{N}}V(x)|u_{n}|^{p-2}u_{n}\phi dx-\int_{\mathbb{R}^{N}}g(x,u_{n})\phi dx=0.
\end{align*}
In the above equation, taking $\phi=|u_{n}|^{2k_{0}}u_{n}$, where $2k_{0}=\frac{Np}{N-p}(1+\alpha)-q>0$. Then, by $(V_{1}), (g_{1})$ and $(h_{0})$, there exists constants $C_{1}>0$, $C>0$, such that
\begin{align*}
C_{1}\int_{\mathbb{R}^{N}}|\nabla u_{n}|^{p}u_{n}^{2k_{0}+\alpha p}dx+\int_{\mathbb{R}^{N}}V(x)|u_{n}|^{2k_{0}+p}dx
\leq C\int_{\mathbb{R}^{N}}|u_{n}|^{2k_{0}+q}dx,
\end{align*}
where $2k_{0}+p<\frac{Np}{N-p}(1+\alpha)$. According to the Sobolev embedding theorem,
we have
\begin{align*}
&(\int_{\mathbb{R}^{N}}|u_{n}|^{(\frac{2k_{0}}{p}+\alpha+1)\frac{Np}{N-p}}dx)^{\frac{N-p}{N}}\\
\leq &C_{1}\int_{\mathbb{R}^{N}}|\nabla u_{n}|^{p}u_{n}^{2k_{0}+\alpha p}dx\\
\leq &C\int_{\mathbb{R}^{N}}|u_{n}|^{2k_{0}+q}dx=:M_{0}\\
= &C\int_{\mathbb{R}^{N}}|u_{n}|^{\frac{Np(1+\alpha)}{N-p}}dx.
\end{align*}
Next, taking $2k_{i}=[2k_{i-1}+(1+\alpha)p]\frac{N}{N-p}-q$, $i\geq1$, then
\begin{align*}
(\int_{\mathbb{R}^{N}}|u_{n}|^{(\frac{2k_{i}}{p}+\alpha+1)\frac{Np}{N-p}}dx)^{\frac{N-p}{N}}
\leq C\int_{\mathbb{R}^{N}}|u_{n}|^{2k_{i}+q}dx=: M_{i}.
\end{align*}
By iterating in sequence, there exists $C_{i}>0$ such that
$$\|u_{n}\|_{L^{2k_{i}+q}}\leq C_{i}^{\frac{1}{2k_{i}+q}}(M_{0})^{(\frac{N}{N-p})^{i}\cdot\frac{1}{2k_{i}+q}}.$$
It should be pointed out that
\begin{align*}
2k_{i}+q=2k_{0}\cdot(\frac{N}{N-p})^{i}+[\frac{Np(1+\alpha)}{N-p}-q]
\cdot\sum^{i-1}_{t=1}(\frac{N}{N-p})^{t}+\frac{Np(1+\alpha)}{N-p},
\end{align*}
Here, when $i=1$, setting $\sum^{i-1}_{t=1}(\frac{N}{N-p})^{t}=0$. Then for $i\geq1$, we have $$\frac{(\frac{N}{N-p})^{i}}{2k_{i}+q}<\frac{1}{2k_{0}}.$$
Thus, $2k_{i}+q\rightarrow\infty$ as $i\rightarrow\infty$. Therefore, $\|u_{n}\|_{L^{\infty}}\leq C$, and $\|u\|_{L^{\infty}}\leq C$ by weakly convergence.\\
\textbf{Step 3:} It should be prove that the weak limit $u\in W^{1,p}_{V}(\mathbb{R}^{N})\cap L^{\infty}(\mathbb{R}^{N})$ and satisfies $\langle d\mathcal{J}(u),\phi\rangle=0$.\\
In (3.3), by choosing $\phi=\psi e^{-Ku_{n}}$ with $\psi\in C^{\infty}_{0}(\mathbb{R}^{N})$, $\psi\geq0$, we have
\begin{align*}
&\nu_{n}\int_{\mathbb{R}^{N}}(|\nabla u_{n}|^{r-2}\nabla u_{n}\nabla\psi-K|\nabla u_{n}|^{r}\psi+|u_{n}|^{r-2}u_{n}\psi)e^{-Ku_{n}}dx\\
&+\int_{\mathbb{R}^{N}}[\frac{1}{p}A_{t}(x,u_{n})-KA(x,u_{n})]|\nabla u_{n}|^{p}\psi e^{-Ku_{n}}dx\\
&+\int_{\mathbb{R}^{N}}A(x,u_{n})|\nabla u_{n}|^{p-2}\nabla u_{n}\nabla\psi e^{-Ku_{n}}dx+\int_{\mathbb{R}^{N}}V(x)|u_{n}|^{p-2}u_{n}\psi e^{-Ku_{n}}dx\\
=&\int_{\mathbb{R}^{N}}g(x,u_{n})\psi e^{-Ku_{n}}dx.
\end{align*}
At this point, since $\{u_{n}\}$ weakly converges to $u$ in $W^{1,p}_{V}(\mathbb{R}^{N})$, and by the Dominated convergence theorem, we have
\begin{align*}
\int_{\mathbb{R}^{N}}A(x,u_{n})|\nabla u_{n}|^{p-2}\nabla u_{n}\nabla\psi e^{-Ku_{n}}dx
&\rightarrow\int_{\mathbb{R}^{N}}A(x,u)|\nabla u|^{p-2}\nabla u\nabla\psi e^{-Ku}dx,\\
\int_{\mathbb{R}^{N}}V(x)|u_{n}|^{p-2}u_{n}\psi e^{-Ku_{n}}dx&\rightarrow\int_{\mathbb{R}^{N}}V(x)|u|^{p-2}u\psi e^{-Ku}dx,\\
\int_{\mathbb{R}^{N}}g(x,u_{n})\psi e^{-Ku_{n}}dx&\rightarrow\int_{\mathbb{R}^{N}}g(x,u)\psi e^{-Ku}dx.
\end{align*}
Choosing $K>0$ large enough, such that $\frac{1}{p}A_{t}(x,u_{n})-KA(x,u_{n})\leq0$. By Fatou's Lemma, we get
\begin{align*}
&\liminf_{n\rightarrow\infty}\int_{\mathbb{R}^{N}}[\frac{1}{p}A_{t}(x,u_{n})-KA(x,u_{n})]|\nabla u_{n}|^{p}\psi e^{-Ku_{n}}dx\\
\leq&
\int_{\mathbb{R}^{N}}[\frac{1}{p}A_{t}(x,u)-KA(x,u)]|\nabla u|^{p}\psi e^{-Ku}dx,
\end{align*}
Then we deduce that
\begin{align*}
0&\leq\int_{\mathbb{R}^{N}}[\frac{1}{p}A_{t}(x,u)-KA(x,u)]|\nabla u|^{p}\psi e^{-Ku}dx
+\int_{\mathbb{R}^{N}}A(x,u)|\nabla u|^{p-2}\nabla u\nabla\psi e^{-Ku}dx\\
&\qquad+\int_{\mathbb{R}^{N}}V(x)|u|^{p-2}u\psi e^{-Ku}dx
-\int_{\mathbb{R}^{N}}g(x,u)\psi e^{-Ku}dx.
\end{align*}
Let $\varphi\in C^{\infty}_{0}(\mathbb{R}^{N})$, $\varphi\geq0$. Choose a sequence of non-negative functions $\psi_{n}\in C^{\infty}_{0}(\mathbb{R}^{N})$ such that $\psi_{n}\rightarrow\varphi e^{Ku}$ in $W^{1,p}_{V}(\mathbb{R}^{N})$, and $\psi_{n}$ is uniformly bounded in $L^{\infty}(\mathbb{R}^{N})$. Let $\psi=\psi_{n}$ in the above inequality, we can obtain that for all $\varphi\in C^{\infty}_{0}(\mathbb{R}^{N})$, \,$\varphi\geq0$,
\begin{align*}
\int_{\mathbb{R}^{N}}A(x,u)|\nabla u|^{p-2}\nabla u\nabla\varphi&+\frac{1}{p}\int_{\mathbb{R}^{N}}A_{t}(x,u)\varphi|\nabla u|^{p}dx\\
&+\int_{\mathbb{R}^{N}}V(x)|u|^{p-2}u\varphi dx-\int_{\mathbb{R}^{N}}g(x,u)\varphi \geq0.
\end{align*}
Similarly, by choosing $\phi=\psi e^{Ku_{n}}$, and repeating the above analysis, we know that for any $\varphi\in C^{\infty}_{0}(\mathbb{R}^{N})$, \,$\varphi\geq0$,
\begin{align*}
\int_{\mathbb{R}^{N}}A(x,u)|\nabla u|^{p-2}\nabla u\nabla\varphi&+\frac{1}{p}\int_{\mathbb{R}^{N}}A_{t}(x,u)\varphi|\nabla u|^{p}dx\\
&+\int_{\mathbb{R}^{N}}V(x)|u|^{p-2}u\varphi dx-\int_{\mathbb{R}^{N}}g(x,u)\varphi dx=0.
\end{align*}
\textbf{Step 4:} In this step, we prove that $u_{n}$ strongly converges to $u$ in $W^{1,p}_{V}(\mathbb{R}^{N})$.\\
According to the analysis of Step 1, $u_{n}$ strongly converges to $u$ in $L^{s}(\mathbb{R}^{N})$, where $s\in[p,(1+\alpha)p^{*})$. Then, following from the Dominated convergence theorem, the H\"{o}lder inequality and $(g_{1})$, we obtain that
$$\int_{\mathbb{R}^{N}}g(x,u_{n})u_{n}dx\rightarrow\int_{\mathbb{R}^{N}}g(x,u)udx.$$
And since in Step 3, we have proven that $\langle dI(u),\phi\rangle=0$ for any $\phi\in W^{1,p}_{V}(\mathbb{R}^{N})\cap L^{\infty}(\mathbb{R}^{N})$. Setting $\phi=u$, we have
\begin{align*}
\int_{\mathbb{R}^{N}}A(x,u)|\nabla u|^{p}dx&+\frac{1}{p}\int_{\mathbb{R}^{N}}A_{t}(x,u)u|\nabla u|^{p}dx\\
&+\int_{\mathbb{R}^{N}}V(x)|u|^{p}dx-\int_{\mathbb{R}^{N}}g(x,u)u dx=0.
\end{align*}
And following from the assumption $I_{\nu_{n}}^{'}(u_{n})=0$, we obtain that
\begin{align*}
\nu_{n}\int_{\mathbb{R}^{N}}(|\nabla u_{n}|^{r}+|u_{n}|^{r})dx
+&\int_{\mathbb{R}^{N}}A(x,u_{n})|\nabla u_{n}|^{p}dx+\frac{1}{p}\int_{\mathbb{R}^{N}}A_{t}(x,u_{n})u_{n}|\nabla u_{n}|^{p}dx\\
&+\int_{\mathbb{R}^{N}}V(x)|u_{n}|^{p} dx-\int_{\mathbb{R}^{N}}g(x,u_{n})u_{n}dx=0.
\end{align*}
Then, by the uniform boundedness of $\|u_{n}\|_{L^{\infty}}$, and $(V_{1}), (h_{0}), (h_{1})$, we have
\begin{align*}
&\nu_{n}\int_{\mathbb{R}^{N}}(|\nabla u_{n}|^{r}+|u_{n}|^{r})dx\rightarrow0,\\
&\int_{\mathbb{R}^{N}}A(x,u_{n})|\nabla u_{n}|^{p}dx+\frac{1}{p}\int_{\mathbb{R}^{N}}A_{t}(x,u_{n})u_{n}|\nabla u_{n}|^{p}dx\\
\rightarrow&\int_{\mathbb{R}^{N}}A(x,u)|\nabla u|^{p}dx+\frac{1}{p}\int_{\mathbb{R}^{N}}A_{t}(x,u)u|\nabla u|^{p}dx,\\
&\int_{\mathbb{R}^{N}}V(x)|u_{n}|^{p} dx\rightarrow\int_{\mathbb{R}^{N}}V(x)|u|^{p}dx.
\end{align*}
Therefore, $u_{n}$ strongly converges to $u$ in $W^{1,p}_{V}(\mathbb{R}^{N})$.
\end{proof}

In order to prove the multiplicity of solutions of equation (\ref{PL}) by using the symmetric mountain pass theorem, the following lemma is also required.
\begin{lemma}\label{LEM3.5}\rm{For $2<p<N$, and fixed $\nu\in(0,1]$. Assume that $(V_{1}),\,(V_{2})$, $(h_{0})$ and $(g_{0})-(g_{2})$ hold. Then the perturbation functional $I_{\nu}(u)$ satisfies Palais-Smale condition in the perturbation space $X$.}
\end{lemma}
\begin{proof} Let $\{u_{n}\}$ be any Palais-Smale sequence of $I_{\nu}(u)$, that is,
\begin{align*}
I_{\nu}(u_{n})\leq C, ~~\lim_{n\rightarrow\infty}\|I_{\nu}(u_{n})\|_{X^{'}}\rightarrow0.
\end{align*}
Following from the Step 1 in the proof of Lemma \ref{LEM3.4}, we know that $\{u_{n}\}$ is bounded in
$X:=W^{1,r}(\mathbb{R}^{N})\cap W^{1,p}_{V}(\mathbb{R}^{N})$. Next, we will discuss the following two cases separately: the first case is $\int_{\mathbb{R}^{N}}(|\nabla u_{n}|^{r}+|u_{n}|^{r})dx\rightarrow0$. Then we can deduce that $\int_{\mathbb{R}^{N}}|u_{n}|^{s}\rightarrow0$ for all $s\in(r,\frac{Np(1+\alpha)}{N-p})$. Since $\{u_{n}\}$ is bounded in $L^{p}(\mathbb{R}^{N})$, then we have $\int_{\mathbb{R}^{N}}|u_{n}|^{\tilde{s}}dx\rightarrow0$ for all $\tilde{s}\in(p, \frac{Np(1+\alpha)}{N-p})$. By $(g_{1})$ and the interpolation inequality, we have $\int_{\mathbb{R}^{N}}g(x,u_{n})u_{n}dx\rightarrow0$. Moreover, we also get
\begin{align*}
&\nu\int_{\mathbb{R}^{N}}(|\nabla u_{n}|^{r}+|u_{n}|^{r})dx
+\int_{\mathbb{R}^{N}}A(x,u_{n})|\nabla u_{n}|^{p}dx+\frac{1}{p}\int_{\mathbb{R}^{N}}A_{t}(x,u_{n})u_{n}|\nabla u_{n}|^{p}dx\\
&~~+\int_{\mathbb{R}^{N}}V(x)|u_{n}|^{p} dx=\int_{\mathbb{R}^{N}}g(x,u_{n})u_{n} dx+\langle I{'}_{\nu}(u_{n}),u_{n}\rangle\rightarrow0.
\end{align*}
Thus, by the condition $(h_{0})$, we can deduce that $\{u_{n}\}$ strongly converges to zero in $X$. The seconde case: $\int_{\mathbb{R}^{N}}(|\nabla u_{n}|^{r}+|u_{n}|^{r})dx\rightarrow\eta_{0}>0$. Then we obtain that
\begin{align*}
o(1)&=\langle I^{'}_{\nu}(u_{n})-I^{'}_{\nu}(u_{m}),u_{n}-u_{m}\rangle\\
&=\nu
\int_{\mathbb{R}^{N}}(|\nabla u_{n}|^{r-2}\nabla u_{n}-|\nabla u_{m}|^{r-2}\nabla u_{m})(\nabla u_{n}-\nabla u_{m})dx\\
&~~+\nu
\int_{\mathbb{R}^{N}}(|u_{n}|^{r-2}u_{n}-|u_{m}|^{r-2}u_{m})(u_{n}-u_{m})dx\\
&~~+\int_{\mathbb{R}^{N}}[A(x,u_{n})|\nabla u_{n}|^{p-2}\nabla u_{n}-A(x,u_{m})|\nabla u_{m}|^{p-2}\nabla u_{m}](\nabla u_{n}-\nabla u_{m})dx\\
&~~+\frac{1}{p}\int_{\mathbb{R}^{N}}[A_{t}(x,u_{n})|\nabla u_{n}|^{p}-A_{t}(x,u_{m})|\nabla u_{m}|^{p}](u_{n}-u_{m})dx\\
&~~+\int_{\mathbb{R}^{N}}V(x)[|u_{n}|^{p-2}u_{n}-|u_{m}|^{p-2}u_{m}](u_{n}-u_{m})dx\\
&~~-\int_{\mathbb{R}^{N}}[g(x,u_{n})-g(x,u_{m})](u_{n}-u_{m})dx\\
&=: P_{1}+P_{2}+P_{3}+P_{4}+P_{5}+P_{6}.
\end{align*}
Next, we will analyze the above six items separately. In this process, the following basic inequality will be used: for any $\xi,\zeta\in\mathbb{R}^{N}$, there exists some $C_{\gamma}>0$ such that
$(|\xi|^{\gamma-2}\xi-|\zeta|^{\gamma-2}\zeta)(\xi-\zeta)\geq C_{\gamma}|\xi-\zeta|^{\gamma}$, $\gamma\geq2$. The details are as follows: since $p>2$, and $r=(1+\alpha)p>3$, we have
\begin{align*}
P_{1}+P_{2}\geq C_{r}\int_{\mathbb{R}^{N}}|\nabla u_{n}-\nabla u_{m}|^{r}dx+|u_{n}-u_{m}|^{r}dx.
\end{align*}
$$P_{5}\geq\int_{\mathbb{R}^{N}}V(x)|u_{n}-u_{m}|^{p}dx.$$
According to $(g_{1})$ and the Dominated Convergence Theorem, we have $P_{6}\rightarrow0$. Following from $(h_{0})$, the Dominated Convergence Theorem and the basic inequality mentioned above,  we obtain that
\begin{align*}
P_{3}&=\int_{\mathbb{R}^{N}}a_{1}(|\nabla u_{n}|^{p-2}\nabla u_{n}-|\nabla u_{m}|^{p-2}\nabla u_{m})(\nabla u_{n}-\nabla u_{m})dx\\
&\qquad+\int_{\mathbb{R}^{N}}a_{2}(|u_{n}|^{\alpha p}-|u_{m}|^{\alpha p})|\nabla u_{n}|^{p-2}\nabla u_{n}(\nabla u_{n}-\nabla u_{m})dx\\
&\qquad+\int_{\mathbb{R}^{N}}a_{2}|u_{m}|^{\alpha p}(|\nabla u_{n}|^{p-2}\nabla u_{n}-|\nabla u_{m}|^{p-2}\nabla u_{m})(\nabla u_{n}-\nabla u_{m})dx\\
&\geq\int_{\mathbb{R}^{N}}a_{1}|\nabla u_{n}-\nabla u_{m}|^{p}dx
+\int_{\mathbb{R}^{N}}a_{2}|u_{m}|^{\alpha p}|\nabla u_{n}-\nabla u_{m}|^{p}dx+o(1).
\end{align*}
According to the Sobolev embedding theorem, the H\"{o}lder inequality and the Dominated Convergence Theorem, we have
\begin{align*}
P_{4}&=\frac{1}{p}\int_{\mathbb{R}^{N}}a_{2}(|u_{n}|^{\alpha p-1}-|u_{m}|^{\alpha p-1})|\nabla u_{n}|^{p}(u_{n}-u_{m})dx\\
&\qquad+\frac{1}{p}\int_{\mathbb{R}^{N}}a_{2}|u_{m}|^{\alpha p-1}(|\nabla u_{n}|^{p}-|\nabla u_{m}|^{p})(u_{n}-u_{m})dx\\
&=
o(1).
\end{align*}
In summary, we have
\begin{align*}
o(1)&=\langle I^{'}_{\nu}(u_{n})-I^{'}_{\nu}(u_{m}),u_{n}-u_{m}\rangle\\
&\geq C_{r}\int_{\mathbb{R}^{N}}|\nabla u_{n}-\nabla u_{m}|^{r}dx+\int_{\mathbb{R}^{N}}|u_{n}-u_{m}|^{r}dx\\
&\qquad+\int_{\mathbb{R}^{N}}a_{1}|\nabla u_{n}-\nabla u_{m}|^{p}dx
+\int_{\mathbb{R}^{N}}a_{2}|u_{m}|^{\alpha p}|\nabla u_{n}-\nabla u_{m}|^{p}dx\\
&\qquad
+\int_{\mathbb{R}^{N}}V(x)|u_{n}-u_{m}|^{p}dx+o(1).
\end{align*}
Thus, it can be inferred that the PS sequence $\{u_{n}\}$ is the Cauchy sequence in $X:=W^{1,r}(\mathbb{R}^{N})\cap W^{1,p}_{V}(\mathbb{R}^{N})$. And, by the completeness of $X$, the Palais-Smale condition is hold.
\end{proof}

Next, we will use the symmetric mountain pass theorem to prove Theorem \ref{THM1.1}.

\textbf{The proof of Theorem \ref{THM1.1}.} It should be pointed out that the workspace of the perturbation functional $I_{\nu}(u)$ is $X:=W^{1,p}_{V}(\mathbb{R}^{N})\cap W^{1,r}(\mathbb{R}^{N})$, where $r=(1+\alpha)p$, \,$\alpha p>1$. First of all, for fixed $\nu\in(0,1]$, following from the conditions in the Introduction, we know that the perturbation functional $I_{\nu}(u)$ and the primitive functional $I(u)$ are even. Then, according to Proposition \ref{PRO3.2} and Lemma \ref{LEM3.5}, we know that $I_{\nu}(u)$ is $C^{1}$ and satisfies the Palais-Smale condition. Next, we will use the symmetric mountain pass theorem \cite{R65CBMS} mentioned in Section 2, and the following three steps to prove that the quasi-linear equation (\ref{PL}) has infinite solutions.\\
\textbf{Step 1:} under the assumptions mentioned in Section 1, we will verify that the perturbation functional $I_{\nu}(u)$ satisfies the two conditions of the symmetric mountain pass theorem mentioned in Section 2.\\
\textbf{Step 2:} we will prove that the minimax value $c_{j}(\nu)$ related to the perturbation functional $I_{\nu}(u)$ is uniformly bounded, where the definition of $c_{j}(\nu)$ will be given in the following text. At this point, according to Lemma \ref{LEM3.4}, we know that the primitive functional $I(u)$ has a critical point $u_{j}$ which satisfies $I(u)=c_{j}:=\lim_{\nu\rightarrow0}c_{j}(\nu)$.\\
\textbf{Step 3:} on the basis of Step 2, by estimating the nonlinear term of the perturbation functional $I_{\nu}(u)$, it can be proved through the existence of multiple solutions to the classical $p$-Laplacian equation that $c_{j}\rightarrow+\infty$ as $j\rightarrow\infty$.\\
The details are as follows:

\textbf{Step 1:} first, we need to prove that $I_{\nu}(u)$ satisfies $(1)$ in the symmetric mountain pass theorem mentioned in Section 2.\\
For each finite dimensional subsequence $E_{j}$ in $X$, $\dim E_{j}=j$. Choose $w\in E_{j}$. Following from $(g_{2})$, we know that for any $\epsilon>0$, there exists $0<\eta_{\epsilon}(x)\in L^{\infty}(\mathbb{R}^{N})$ such that for almost everywhere $x\in\mathbb{R}^{N}$, $G(x,t)\geq\eta_{\epsilon}(x)|t|^{\mu}$ if $|t|\geq \epsilon$. Then,  when $t\rightarrow\infty$, we have
\begin{align*}
I_{\nu}(tw)\leq&\frac{t^{r}\nu}{2}\int_{\mathbb{R}^{N}}(|\nabla w|^{r}+|w|^{r})dx
+\frac{t^{p}}{p}\int_{\mathbb{R}^{N}}a_{1}|\nabla w|^{p}dx
+\frac{t^{(\alpha+1)p}}{p}\int_{\mathbb{R}^{N}}a_{2}|w|^{\alpha p}|\nabla w|^{p}dx\\
&\qquad+t^{p}\int_{\mathbb{R}^{N}}V(x)|w|^{p}dx
-\frac{t^{\mu}}{\mu}\int_{\mathbb{R}^{N}}\eta_{\epsilon}(x)|w|^{\mu}dx\\
&\rightarrow-\infty.
\end{align*}
Thus, there exists $R_{j}=R(E_{j})>0$ large enough, such that for $u\in E_{j}/B_{R_{j}}$, $I_{\nu}(u)<0$ holds.\\
Seconde, it need to prove that $I_{\nu}(u)$ satisfies $(2)$ in the symmetric mountain pass theorem mentioned in Section 2.\\
Following from the condition $(g_{1})$ mentioned in Section 1, we know that for any $\epsilon>0$, there exists $C_{\epsilon}>0$ satisfies $$\int_{\mathbb{R}^{N}}G(x,u)dx\leq\epsilon\frac{\int_{\mathbb{R}^{N}}|u|^{p}dx}{p}+C_{\epsilon}
\frac{\int_{\mathbb{R}^{N}}|u|^{q}dx}{q}.$$
At this point, by the condition $(h_{0})$, there exists some $m>0$, such that
\begin{align*}
I_{\nu}(u)&=\frac{\nu}{2}\int_{\mathbb{R}^{N}}(|\nabla u|^{r}+|u|^{r})dx+\frac{1}{p}\int_{\mathbb{R}^{N}}A(x,u)|\nabla u|^{p}+V(x)|u|^{p}dx-\int_{\mathbb{R}^{N}}G(x,u)dx\\
&\geq\frac{m}{p}\int_{\mathbb{R}^{N}}|u|^{\alpha p}|\nabla u|^{p}dx+\frac{m}{p}\int_{\mathbb{R}^{N}}(|\nabla u|^{p}+V(x)|u|^{p})dx-C\frac{\int_{\mathbb{R}^{N}}|u|^{q}dx}{q}+o(1).
\end{align*}
Next, we denote
$$\mathbf{U}:=\{u\in X\,|\, u\neq0,\, m[\int_{\mathbb{R}^{N}}|u|^{\alpha p}|\nabla u|^{p}+(|\nabla u|^{p}+V(x)|u|^{p})dx]\geq C\int_{\mathbb{R}^{N}}|u|^{q}dx\}.$$
Then, for $u\in\partial\mathbf{U}\cap E^{\bot}_{j}$, according to the H\"{o}lder inequality and the continuous  embedding relationship from $W^{1,p}_{V}(\mathbb{R}^{N})$ to $L^{p}(\mathbb{R}^{N})$, we have
\begin{align*}
\int_{\mathbb{R}^{N}}|u|^{q}dx&\leq(\int_{\mathbb{R}^{N}}|u|^{p}dx)^{1-\theta}
(\int_{\mathbb{R}^{N}}|u|^{\frac{Np(1+\alpha)}{N-p}}dx)^{\theta}\\
&\leq C_{1}(\int_{\mathbb{R}^{N}}|\nabla u|^{p}+V(x)|u|^{p}dx)^{1-\theta}\cdot
(\int_{\mathbb{R}^{N}}|\nabla u|^{p}u^{\alpha p}dx)^{\frac{N\theta}{N-p}}\\
&\leq C_{2}(\int_{\mathbb{R}^{N}}|u|^{q}dx)^{1-\theta+\frac{N\theta}{N-p}},
\end{align*}
where $\theta=\frac{(q-p)(N-p)}{(N\alpha+p)p}$, $1-\theta+\frac{N\theta}{N-p}=\frac{N\alpha+q}{N\alpha+p}>1$. Thus, for $u\in\partial\mathbf{U}\cap E^{\bot}_{j}$, there exists $\tilde{C}>0$ such that
$$\int_{\mathbb{R}^{N}}|u|^{q}dx\geq\tilde{C}.$$
So, for $u\in\partial\mathbf{U}\cap E^{\bot}_{j}$, we have
$$I(u)\geq(\frac{q-p}{pq})C\int_{\mathbb{R}^{N}}|u|^{q}dx\geq C\tilde{C}=:\alpha. ~~(\alpha ~is ~ independent~ to~ \nu).$$
\textbf{Step 2:} firstly, we denote the minimax value
$$c_{j}(\nu):=\inf_{B\in\Gamma_{j}}\sup_{u\in B}I_{\nu}(u),$$
where
$$\Gamma_{j}:=\{B\,|\,B=h(\overline{D_{k}/Y}),\,h\in C(D_{k}, X) ~is ~odd, ~and ~h=id ~on~ \partial B_{R_{k}}\cap E_{k}\}.$$
Note that $D_{k}:=B_{R_{k}}\cap E_{k}$, $k\geq j$, $\gamma(Y)\leq k-j$. $\gamma(\cdot)$ denote as the genus of symmetric closed set, refer to \cite{R65CBMS}. And
$$Y\in\Sigma(X):=\{A/\{0\}\,|\, A\subset X,\,A=-A ~is~ closed ~set\}.$$
Following from the intersection theorem in \cite{R65CBMS}, we can obtain that $B\cap\partial\mathbf{U}\cap E_{j}^{\bot}\neq\emptyset$. Then we have $c_{j}(\nu)\geq\alpha$. Now, for fixed $j$, choosing $h=id$, $Y=\emptyset$. Set $B_{0}=id(D_{k})=\overline{B_{R_{k}}\cap E_{k}}\in\Gamma_{j}$. Following from the above analysis, we can deduce that there exists $\beta_{j}>0$(independent to $\nu$) satisfies
$$\alpha\leq c_{j}(\nu)\leq\sup_{u\in B_{0}}I_{\nu}(u)\leq\sup_{u\in B_{0}}I_{1}(u):=\beta_{j}.$$
That is, $c_{j}(\nu)\in [\alpha, \beta_{j}]$. Then, according to Lemma \ref{LEM3.4}, we know that the primitive functional $I(u)$ has a critical point $u_{j}$, which satisfies
$$I(u_{j})=c_{j}:=\lim_{v\rightarrow0}c_{j}(\nu).$$
\textbf{Step 3:} To prove $c_{j}\rightarrow+\infty$ as $j\rightarrow\infty$.\\
Following from the conditions mentioned in the Introduction, there exists $\epsilon_{0}>0$ such that
$G(x,t)\leq \epsilon_{0}|t|^{q}-V(x)|t|^{(1+\alpha) p}+C_{\epsilon_{0}}|t|^{q}$. Then we have
\begin{align*}
I_{\nu}(u)&\geq\frac{\nu}{2}(\int_{\mathbb{R}^{N}}|\nabla u|^{r}+|u|^{r}dx)+\frac{a_{1}}{p}\int_{\mathbb{R}^{N}}|\nabla u|^{p}dx+\frac{a_{2}}{p(1+\alpha)^{p}}\int_{\mathbb{R}^{N}}|\nabla |u|^{1+\alpha}|^{p}dx
\\&\qquad+\frac{1}{p}\int_{\mathbb{R}^{N}}V(x)|u|^{p}dx
-\epsilon_{0}\int_{\mathbb{R}^{N}}|u|^{p}dx+\int_{\mathbb{R}^{N}}V(x)|u|^{(1+\alpha)p}dx
-C_{\epsilon}\int_{\mathbb{R}^{N}}|u|^{q}dx\\
&\geq \frac{a}{p}\int_{\mathbb{R}^{N}}|\nabla |u|^{1+\alpha}|^{p}dx+\frac{a}{p}\int_{\mathbb{R}^{N}}V(x)|u|^{(1+\alpha)p}dx
-\frac{C(1+\alpha)}{q}\int_{\mathbb{R}^{N}}(|u|^{1+\alpha})^{\frac{q}{1+\alpha}}dx\\
&=\frac{a}{p}\int_{\mathbb{R}^{N}}|\nabla w|^{p}dx+\frac{a}{p}\int_{\mathbb{R}^{N}}V(x)|w|^{p}dx
-\frac{C(1+\alpha)}{q}\int_{\mathbb{R}^{N}}|w|^{\frac{q}{1+\alpha}}dx:=L(w),
\end{align*}
where $w=H(u):=|u|^{\alpha}u\in W^{1,p}(\mathbb{R}^{N})$, $\frac{q}{1+\alpha}\in(p,\frac{Np}{N-p})$. According to \cite{LT13NA}, we know that $L(w)$ has an unbounded sequence of critical values. Therefore,
$I(u_{j})=c_{j}\rightarrow+\infty$ as $j\rightarrow+\infty$. $\blacksquare$
\begin{remark}\label{REk3.6}\rm{ When the assumptions $(V_{2})$, $(h_{0})$ in the Introduction are replaced by the following $(V_{3})$ and $(h_{1})$, the above conclusion still holds. The reason is that under condition $(V_{3})$, $W^{1,p}_{V}(\mathbb{R}^{N})$ is compactly embedded into $L^{s}(\mathbb{R}^{N})$, $s\in(p,p^{*})$. Moreover, under this condition, the compactness of the Palais-Smale sequence holds, which can refer \cite{LT13NA, LZ11JMAA}.
\begin{description}
  \item[$(V_{3})$.] There exists $r>0$, such that for any $b>0$,
  $$\lim_{|y|\rightarrow\infty}meas(\{x\in\mathbb{R}^{N},\,V(x)\leq b\}\cap B_{r}(y))=0.$$
  \item[$(h_{1})$.] $A(x,t)=a_{1}(x)+a_{2}(x)|t|^{\alpha p}$, \,$\alpha p>1$. where $a_{1}(x)$, $a_{2}(x)\in L^{\infty}(\mathbb{R}^{N})$ and are positive.
\end{description}
}
\end{remark}
\section{DATA AVAILABILITY STATEMENT}
The data supporting this study's findings are available from the corresponding author upon reasonable request.
\section{CONFLICT OF INTEREST STATEMENT}
The authors declare no conflicts of interest.
\section{FUNDING DECLARATION}
This work was supported by the Young Project of Henan Natural Science Foundation(Grant No.252300420928)





\end{document}